\documentclass[12pt,reqno]{amsart}

\usepackage{amssymb,amsmath,graphicx,amsfonts,euscript}
\usepackage{color}
\usepackage{cite}

\newtheorem{theorem}{Theorem}[section]

\newtheorem{definition}[theorem]{Definition}

\newtheorem{lemma}[theorem]{Lemma}

\numberwithin{equation}{section}
\begin{document}
\title[Regularity Criteria for the Navier--Stokes/Poisson--Nernst--Planck System]{Logarithmical  regularity criteria in terms of pressure for the three dimensional nonlinear dissipative system modeling
electro-diffusion}

\author[Jihong Zhao]{Jihong Zhao}

\address{College of Science, Northwest A\&F
University,  Yangling,  Shaanxi  712100,  China}

\email{jihzhao@163.com.}

\vskip .2in
\begin{abstract}
In this paper, logarithmically improved regularity criteria for the  Navier--Stokes/Poisson--Nernst--Planck system are established in terms of both the pressure and the gradient of pressure in the homogeneous Besov space.
\end{abstract}

\maketitle

\section{Introduction}

In this paper, we study the following Navier--Stokes/Poisson--Nernst--Planck system modeling electro-diffusion, which is governed  by transport, and Lorentz force coupling between the Navier--Stokes equations of an incompressible fluid and the transported Poisson--Nernst--Planck equations of a binary diffuse charge densities:
\begin{equation}\label{eq1.1}
\begin{cases}
  \partial_{t} u+(u\cdot\nabla) u-\Delta
  u+\nabla \Pi=\Delta
  \Psi\nabla\Psi,\ \ & x\in\mathbb{R}^{3},\ t>0,\\
  \nabla\cdot u=0,\ \ & x\in\mathbb{R}^{3},\ t>0,\\
  \partial_{t} v+(u\cdot \nabla)
  v=\nabla\cdot(\nabla v-v\nabla \Psi),\ \ & x\in\mathbb{R}^{3},\ t>0,\\
  \partial_{t} w+(u\cdot \nabla)
  w=\nabla\cdot(\nabla w+w\nabla \Psi),\ \ & x\in\mathbb{R}^{3},\ t>0,\\
  \Delta \Psi=v-w,\ \ & x\in\mathbb{R}^{3},\ t>0,\\
   (u, v, w)|_{t=0}=(u_0, v_0, w_0), \ \ & x\in\mathbb{R}^{3},
\end{cases}
\end{equation}
where  $u=(u_{1},u_{2}, u_{3})$ is the velocity field,   $\Pi$ is the
pressure,  $\Psi$
is the electric potential,  $v$ and $w$ are the densities of binary diffuse negative and positive charges (e.g., ions),
respectively.    The right-hand side term in the momentum equations is the Lorentz force\footnote{Assume that the average ion velocities is small (compared to the speed of light), so that the contribution to the Lorentz force from the magnetic field is negligible.}, which exhibits
$\Delta\Psi\nabla\Psi=\nabla\cdot\sigma$,
where the electric stress $\sigma$ is a rank one tensor plus a
pressure, for $i,j=1,2,3$,
\begin{equation}\label{eq1.2}
  [\sigma]_{ij}=\big(\nabla\Psi\otimes\nabla\Psi-\frac{1}{2}|\nabla\Psi|^{2}I\big)_{ij}
  =\partial_{x_{i}}\Psi\partial_{x_{j}}\Psi-\frac{1}{2}|\nabla\Psi|^{2}\delta_{ij}.
\end{equation}
Here $\otimes$ denotes the tensor product, $I$ is $3\times3$ identity matrix and  $\delta_{ij}$ is the
Kronecker symbol. The
electric stress $\sigma$ stems from the balance of kinetic energy
with electrostatic energy via the least action principle (cf.
\cite{RLZ07}). For simplicity, we have assumed that the fluid density, viscosity, charge mobility and dielectric constant are unity.
\vskip .1in

System \eqref{eq1.1} was first proposed by Rubinstein \cite{R90}, which is capable of describing electro-chemical and
fluid-mechanical transport throughout the cellular environment,  we
refer the readers to see \cite{BW04}, \cite{J11}, \cite{RLZ07}, \cite{S09}  and the references therein for more details of physical background and applied aspects. Based on Kato's semigroup framework, the local smooth theory of the system \eqref{eq1.1} has been established by Jerome \cite{J02}. In more general situation, the global existence of strong solutions with  small initial data and the local existence of strong solutions with any initial data  in various scaling invariant spaces have been studied by  Zhang and Yin  \cite{ZY15},  Zhao, Deng and Cui \cite{ZDC10}, Zhao and Liu \cite{ZL15}, and  Zhao, Zhang and Liu \cite{ZZL15}. The Prodi--Serrin type blow up criterion and the Beale--Kato--Majda type blow-up criterion for local smooth solutions have been established by Zhao and Bai \cite{ZB16}.
\vskip .1in

Notice that, the first two equations of system \eqref{eq1.1} are the linear momentum equations of incompressible flow, which reduces to the following  Navier--Stokes equations if the flow is charge-free (i.e., $v=w=\Psi=0$):
\begin{equation}\label{eq1.3}
\begin{cases}
  \partial_{t} u+(u\cdot\nabla)u-\Delta
  u+\nabla \Pi=0,\ \ & x\in\mathbb{R}^{3},\ t>0,\\
  \nabla\cdot u=0,\ \ & x\in\mathbb{R}^{3},\ t>0,\\
  u|_{t=0}=u_0,\ \ & x\in\mathbb{R}^{3}.\\
\end{cases}
\end{equation}
In \cite{L34}, Leray constructed a global weak solution $u\in L^{\infty}(0, \infty; L^{2}(\mathbb{R}^{3}))\cap  L^{2}(0, \infty; H^{1}(\mathbb{R}^{3}))$ for arbitrary initial data $u_{0}\in L^{2}(\mathbb{R}^{3})$ and questioned whether the weak solution is regular or smooth for a given smooth and compactly supported initial data $u_{0}$. This problem still remains generally unsolved.  In the celebrated work \cite{S62}, Serrin  introduced the class $L^{q}(0,T;L^{p}(\mathbb{R}^{3}))$ and showed that any weak solution $u$ satisfies
\begin{align}\label{eq1.4}
\int_{0}^{T}\|u(\cdot,t)\|_{L^{p}}^{q}dt<\infty\ \ \ \text{with}\ \ \ \frac{2}{q}+\frac{3}{p}=1\ \ \ \text{and} \ \ \  3<p\leq \infty,
\end{align}
then $u$ is regular on $(0,T]$, while the limit case $p=3$ was covered by Escauriaza et al. \cite{ESS03}. Later on,   Beir\~{a}o da
Veiga \cite{B95} showed that if the gradient of the velocity $\nabla u$ satisfies
\begin{align}\label{eq1.5}
  \int_{0}^{T}\|\nabla u(\cdot,t)\|_{L^{p}}^{q}dt<\infty\ \text{ with
  }\ \frac{2}{q}+\frac{3}{p}= 2\ \ \text{ and }\ \ \frac{3}{2}<p\leq \infty,
\end{align}
then  $u$ is still  regular on $(0,T]$. Since then,  many interesting sufficient conditions were proposed to guarantee the regularity of weak solutions, see  \cite{FJNZ10},  \cite{KOT02}, \cite{KT00} , \cite{M05}, \cite{ZL08} and the references cited there.
\vskip.1in

On the other hand, based on the well-known pressure-velocity relation of the Navier--Stokes equations \eqref{eq1.3} in $\mathbb{R}^{3}$, given by
\begin{equation}\label{eq1.6}
   \Pi=\sum_{i,j=1}^{3}\mathcal{R}_{i}\mathcal{R}_{j}(u_{i}u_{j}),
\end{equation}
 where $\mathcal{R}_{i}$ ($i=1,2,3$) are the Riesz transforms in $\mathbb{R}^{3}$, certain growth conditions in terms of pressure were proposed to ensure the regularity of weak solutions. Chae and Lee \cite{CL01} showed that if  the pressure $\Pi$ satisfies
 \begin{align}\label{eq1.7}
\int_{0}^{T}\|\Pi(\cdot,t)\|_{L^{p}}^{q}dt<\infty\ \ \ \text{with}\ \ \ \frac{2}{q}+\frac{3}{p}<2\ \ \ \text{and} \ \ \  \frac{3}{2}<p\leq \infty,
\end{align}
then the weak solution $u$ is actually smooth. A similar condition in terms of the gradient of  pressure
  \begin{align}\label{eq1.8}
\int_{0}^{T}\|\nabla\Pi(\cdot,t)\|_{L^{p}}^{q}dt<\infty\ \ \ \text{with}\ \ \ \frac{2}{q}+\frac{3}{p}=3\ \ \ \text{and} \ \ \  1<p\leq \infty,
\end{align}
also implies the regularity of weak solution $u$ as shown by Zhou \cite{Z06} and Struwe \cite{S07}.
 These two results were subsequently refined by many authors, we refer to \cite{BG02}, \cite{CZ07}, \cite{FJN08}, \cite{FO08} and \cite{Z05} for more details. Recently,  Fan et al.  \cite{FJNZ10} proved that if the pressure satisfies one of the following conditions:
\begin{align}\label{eq1.9}
\int_{0}^{T}\frac{\|\Pi(\cdot,t)\|_{\dot{B}^{0}_{\infty,\infty}}}{\sqrt{1+\ln(e+\|\Pi(\cdot,t)\|_{\dot{B}^{0}_{\infty,\infty}})}}dt<\infty,
\end{align}
or
\begin{align}\label{eq1.10}
\int_{0}^{T}\frac{\|\nabla \Pi(\cdot,t)\|_{\dot{B}^{0}_{\infty,\infty}}^{\frac{2}{3}}}{\big(1+\ln(e+\|\nabla \Pi(\cdot,t)\|_{\dot{B}^{0}_{\infty,\infty}})\big)^{\frac{2}{3}}}dt<\infty,
\end{align}
then the solution $u$ can be smoothly extended beyond time $T$. Guo and Gala \cite{GG11}  refined the regularity criterion \eqref{eq1.9} to prove that if the pressure satisfies
\begin{align}\label{eq1.11}
\int_{0}^{T}\frac{\|\Pi(\cdot,t)\|_{\dot{B}^{-1}_{\infty,\infty}}^{2}}{1+\ln(e+\|\Pi(\cdot,t)\|_{\dot{B}^{-1}_{\infty,\infty}})}dt<\infty,
\end{align}
then the solution $u$ is smooth up to time $T$.  Due to the scaling invariant property of the Navier--Stokes equations satisfied, these three regularity criteria can be regarded as an eventual forms in terms of pressure and the gradient of pressure.
\vskip.1in

 For the system \eqref{eq1.1},  the global existence of weak solutions of the initial/boundary-value problem has been established by Jerome and Sacco \cite{JS09},  Ryham \cite{R09}  and Schmuck \cite{S09}. Since the Navier--Stokes equations is a subsystem of the system \eqref{eq1.1}, the smoothness of weak solutions of \eqref{eq1.1} is  also unknown.  Some logarithmical improved regularity criteria in terms of the velocity to some more generalized system than \eqref{eq1.1} have been concerned by Fan et al. \cite{FG09}, \cite{FLN13} and \cite{FNZ12}. Motivated by  the regularity criteria \eqref{eq1.9}--\eqref{eq1.11} for the Navier--Stokes equations,  in this paper, we aim at extending these regularity criteria in terms of pressure and the gradient of pressure to  the system \eqref{eq1.1}.
 \vskip.1in

The main results  in this paper now read:

\begin{theorem}\label{th1.1}
Let $(u,v,w)$ be a local smooth solution to the system \eqref{eq1.1} with initial data $u_{0}\in H^{3}(\mathbb{R}^{3})$ and $\nabla\cdot u_{0}=0$, $(v_{0}, w_{0})\in L^{1}(\mathbb{R}^{3})\cap H^{2}(\mathbb{R}^{3})$ and $v_{0}, w_{0}\geq0$. Suppose that the pressure $\Pi$ satisfies one of the following conditions:
\begin{align}\label{eq1.12}
\int_{0}^{T}\frac{\|\Pi(\cdot,t)\|_{\dot{B}^{0}_{\infty,\infty}}}{\sqrt{1+\ln(e+\|\Pi(\cdot,t)\|_{\dot{B}^{0}_{\infty,\infty}})}}dt<\infty,
\end{align}
or
\begin{align}\label{eq1.13}
\int_{0}^{T}\frac{\|\Pi(\cdot,t)\|_{\dot{B}^{-1}_{\infty,\infty}}^{2}}{1+\ln(e+\|\Pi(\cdot,t)\|_{\dot{B}^{-1}_{\infty,\infty}})}dt<\infty.
\end{align}
Then the solution $(u,v,w)$ is smooth up to time $T$.
\end{theorem}

\noindent\textbf{Remark 1.1.}  If  $T_{*}<\infty$ is the maximal existence time of  the solution $(u,v,w)$, then by Theorem \ref{th1.1}, we have
\begin{align*}
\int_{0}^{T_{*}}\frac{\|\Pi(\cdot,t)\|_{\dot{B}^{0}_{\infty,\infty}}}{\sqrt{1+\ln(e+\|\Pi(\cdot,t)\|_{\dot{B}^{0}_{\infty,\infty}})}}dt=\infty
\end{align*}
and
\begin{align*}
\int_{0}^{T_{*}}\frac{\|\Pi(\cdot,t)\|_{\dot{B}^{-1}_{\infty,\infty}}^{2}}{1+\ln(e+\|\Pi(\cdot,t)\|_{\dot{B}^{-1}_{\infty,\infty}})}dt=\infty.
\end{align*}
\vskip0.1in

\noindent\textbf{Remark 1.2.} Since the Sobolev embedding $L^{3}(\mathbb{R}^{3})\rightarrow \dot{B}^{-1}_{\infty,\infty}(\mathbb{R}^{3}) $, the condition \eqref{eq1.13} in Theorem \ref{th1.1} can be superseded by the following  condition
\begin{align*}
\int_{0}^{T}\frac{\|\Pi(\cdot,t)\|_{L^{3}}^{2}}{1+\ln(e+\|\Pi(\cdot,t)\|_{L^{3}})}dt<\infty,
\end{align*}
which still implies the smoothness of the solution $(u,v,w)$ on $(0,T]$.
\vskip.1in

Similarly, logarithmical regularity criterion in terms of the gradient of pressure can be also established as follows.

\begin{theorem}\label{th1.2}
Let $(u,v,w)$ be a local smooth solution to the system \eqref{eq1.1} with initial data $u_{0}\in H^{3}(\mathbb{R}^{3})$ and $\nabla\cdot u_{0}=0$, $(v_{0}, w_{0})\in L^{1}(\mathbb{R}^{3})\cap H^{2}(\mathbb{R}^{3})$ and $v_{0}, w_{0}\geq0$. Suppose that the gradient of pressure $\nabla\Pi$ satisfies
\begin{align}\label{eq1.14}
\int_{0}^{T}\frac{\|\nabla \Pi(\cdot,t)\|_{\dot{B}^{0}_{\infty,\infty}}^{\frac{2}{3}}}{\big(1+\ln(e+\|\nabla \Pi(\cdot,t)\|_{\dot{B}^{0}_{\infty,\infty}})\big)^{\frac{2}{3}}}dt<\infty.
\end{align}
Then the local smooth solution $(u,v,w)$ is smooth up to time $T$.
\end{theorem}

\noindent\textbf{Remark 1.3.}  If  $T_{*}<\infty$ is the maximal existence time of  the solution $(u,v,w)$, then by Theorem \ref{th1.2}, we have
\begin{align*}
\int_{0}^{T_{*}}\frac{\|\nabla \Pi(\cdot,t)\|_{\dot{B}^{0}_{\infty,\infty}}^{\frac{2}{3}}}{\big(1+\ln(e+\|\nabla \Pi(\cdot,t)\|_{\dot{B}^{0}_{\infty,\infty}})\big)^{\frac{2}{3}}}dt=\infty.
\end{align*}
\vskip0.1in

\noindent\textbf{Remark 1.4.}  To prove Theorems \ref{th1.1} and
\ref{th1.2}, the main crux lies in exploiting the relationship of the pressure $\Pi$ and the solution $(u,v,w)$, see Lemma \ref{le2.6} below. Moreover, for simplicity, we just consider initial data $u_{0}\in H^{3}(\mathbb{R}^{3})$, $(v_{0},w_{0}) \in H^{2}(\mathbb{R}^{3})$. As a  matter of fact, one can also prove the same result for initial data  $u_{0}\in H^{s}(\mathbb{R}^{3})$, $(v_{0},w_{0}) \in H^{s-1}(\mathbb{R}^{3})$ with $s>\frac{5}{2}$.
\vskip .1in

 The remainder of the paper is organized as follows. In Section 2,  we first present the Littlewood-Paley decomposition theory and  the definition of  the homogeneous  Besov space, then we list several crucial lemmas used to prove Theorems \ref{th1.1}  and \ref{th1.2}. Section 3 is devoted to derive the crucial a priori estimates of local smooth solutions,  while the last section  is devoted to the proof of Theorems \ref{th1.1} and \ref{th1.2}. Throughout the paper, we denote by $C$ and $C_{i}$ ($i=0,1, \cdots$) the harmless positive constants, which may depend on the initial data and its value may change from line to line, the special dependence will be pointed out explicitly in the text if necessary.

\section{ Preliminaries}

\subsection{Littlewood-Paley theory and Besov space}
We start with the Fourier transform. Let
$\mathcal{S}(\mathbb{R}^{3})$ be the Schwartz class of rapidly
decreasing function, and $\mathcal{S}'(\mathbb{R}^{3})$ of temperate distributions be the dual set of $\mathcal{S}(\mathbb{R}^{3})$. Given  $f\in\mathcal{S}(\mathbb{R}^{3})$, the Fourier transform $\mathcal{F}(f)$ is defined by
$$
  \mathcal{F}(f)(\xi):=\int_{\mathbb{R}^{3}}e^{-2\pi i x\cdot\xi}f(x)dx.
$$
More generally, the Fourier transform of  a tempered distribution $f\in\mathcal{S}'(\mathbb{R}^{3})$ is defined by the dual argument in the standard way.
\vskip.1in

We now introduce a dyadic partition of $\mathbb{R}^{3}$.
Choose $\varphi\in\mathcal{S}(\mathbb{R}^{3})$  such that $\varphi$ is even,
$$
  \text{supp}\varphi:=\mathcal{C}=\{\xi\in\mathbb{R}^{3},\ \frac{3}{4}\leq
|\xi|\leq\frac{8}{3}\}, \ \
\text{and}\ \ \varphi\in[0,1] \ \ \text{on} \ \ \mathcal{C}.
$$
Moreover, there holds
\begin{align*}
  \sum_{j\in\mathbb{Z}}\varphi(2^{-j}\xi)=1, \ \ \ \forall\xi\in\mathbb{R}^{3}\backslash\{0\}.
\end{align*}
Let $h=\mathcal{F}^{-1}\varphi$, where $\mathcal{F}^{-1}$ is the inverse Fourier transform. Then  for any $f\in\mathcal{S}'(\mathbb{R}^{3})$, we define the dyadic blocks
$\Delta_{j}$ and $S_{j}$ as follows:
\begin{align}\label{eq2.1}
 \Delta_{j}f: =2^{3j}\int_{\mathbb{R}^{3}}h(2^{j}y)f(x-y)dy \ \ \ \text{and}\ \ \
  S_{j}f:=\sum_{k\leq j-1}\Delta_{k}f.
\end{align}
 Informally, $\Delta_{j}$ is a
frequency projection to the annulus $\{|\xi|\sim 2^{j}\}$, while
$S_{j}$ is a frequency projection to the
ball $\{|\xi|\leq 2^{j}\}$.
\vskip.1in

Let $\mathcal{P}(\mathbb{R}^{3})$ be the class of all polynomials of $\mathbb{R}^{3}$ and denote by $\mathcal{S}'_{h}(\mathbb{R}^{3}):=\mathcal{S}'(\mathbb{R}^{3})/\mathcal{P}(\mathbb{R}^{3})$ the tempered distributions modulo polynomials. As a consequence, for any $f\in \mathcal{S}'_{h}(\mathbb{R}^{3})$, one has
\begin{equation*}
  f=\sum_{j\in\mathbb{Z}}\Delta_{j}f.
\end{equation*}
\vskip .1in

Next we give the definition of the homogeneous Besov space.

\begin{definition}\label{de2.1}
For $s\in \mathbb{R}$, $1\leq p,r\leq\infty$, the homogeneous Besov space is defined by
\begin{equation*}
  \dot{B}^{s}_{p,r}(\mathbb{R}^{3}):=\Big\{f\in \mathcal{S}'_{h}(\mathbb{R}^{3}):\ \
  \|f\|_{\dot{B}^{s}_{p,r}}<\infty\Big\},
\end{equation*}
where
\begin{equation*}
  \|f\|_{\dot{B}^{s}_{p,r}}:=
  \begin{cases} \left(\sum_{j\in\mathbb{Z}}2^{jsr}\|\Delta_{j}f\|_{L^{p}}^{r}\right)^{\frac{1}{r}}
  \ \ &\text{for}\ \ 1\leq r<\infty,\\
  \sup_{j\in\mathbb{Z}}2^{js}\|\Delta_{j}f\|_{L^{p}}\ \
  &\text{for}\ \
  r=\infty.
 \end{cases}
\end{equation*}
\end{definition}
It is easy to verify that the classical
homogeneous Sobolev space $\dot{H}^{s}(\mathbb{R}^{3})=\dot{B}^{s}_{2,2}(\mathbb{R}^{3})$
which can be characterized by an equivalent norm
$\|f\|_{\dot{H}^{s}}=\|(-\Delta)^{\frac{s}{2}}f\|_{L^{2}}$.

\subsection{Useful analytic tools}
The proofs of Theorems \ref{th1.1} and \ref{th1.2} require some crucial analytic tools related Besov spaces, thus we briefly recall  them here. We denote by $BMO$ the space of
bounded mean oscillation, and $H^{s}_{p}(\mathbb{R}^{3})$
with $s>0$ and $1<p<\infty$ the usual Sobolev space $\{f\in
L^{p}(\mathbb{R}^{3}):\ \  (-\Delta)^{\frac{s}{2}}f\in L^{p}(\mathbb{R}^{3})\}$. If $p=2$, we use $H^{s}(\mathbb{R}^{3})$ instead of $H^{s}_{2}(\mathbb{R}^{3})$.

\begin{lemma}\label{le2.2} {\em(\cite{KT00})}
Let $1<p<\infty$. Then we have
\begin{equation}\label{eq2.2}
  \|fg\|_{L^{p}}
  \leq C\big(\|f\|_{L^{p}}\|g\|_{BMO}+\|f\|_{BMO}\|g\|_{L^{p}}\big)
\end{equation}
for all $f,g \in BMO\cap L^{p}(\mathbb{R}^{3})$ with
the constant $C$ depending only on $p$.
\end{lemma}

\begin{lemma}\label{le2.3} {\em(\cite{KOT02})}
For all $f\in H^{s-1}(\mathbb{R}^{3})$ with $s>\frac{5}{2}$, we have
\begin{equation}\label{eq2.3}
   \|f\|_{BMO}\leq
   C\big(1+\|f\|_{\dot{B}^{0}_{\infty,\infty}}\ln^{\frac{1}{2}}(e+\|f\|_{H^{s-1}})\big).
\end{equation}
\end{lemma}

\begin{lemma}\label{le2.4} {\em(\cite{KP88})}
For $s>1$, we have
\begin{equation}\label{eq2.4}
   \|\nabla^{s}(fg)-f\nabla^{s}g\|_{L^{p}}\leq
   C\big(\|\nabla f\|_{L^{p_{1}}}\|\nabla^{s-1}g\|_{L^{q_{1}}}+\|\nabla^{s}f\|_{L^{p_{2}}}\|g\|_{L^{q_{2}}}\big)
\end{equation}
with $1<p, q_{1}$, $p_{2}<\infty$ such that
$\frac{1}{p}=\frac{1}{p_{1}}+\frac{1}{q_{1}}=\frac{1}{p_{2}}+\frac{1}{q_{2}}$.
\end{lemma}

\begin{lemma}\label{le2.5}{\em (\cite{MGO96})}
Let $1<p<q<\infty$, and let $s=\beta(\frac{q}{p}-1)>0$. Then there
exists a constant $C$ depending only on $\beta$, $p$ and $q$
such that the estimate
\begin{equation}\label{eq2.5}
  \|f\|_{L^{q}}\leq
  C\|f\|_{\dot{B}^{-\beta}_{\infty,\infty}}^{1-\frac{p}{q}}\|(-\Delta)^{\frac{s}{2}}f\|_{L^{p}}^{\frac{p}{q}}
\end{equation}
holds for all $f\in
\dot{H}^{s}_{p}(\mathbb{R}^{3})\cap\dot{B}^{-\beta}_{\infty,\infty}(\mathbb{R}^{3})$.
\end{lemma}

Actually, we shall use the following specific form
of \eqref{eq2.5} by taking $s=\beta=1$, $p=2$ and $q=4$:
\begin{equation}\label{eq2.6}
  \|f\|_{L^{4}}\leq
  C\|f\|_{\dot{B}^{-1}_{\infty,\infty}}^{\frac{1}{2}}\|(-\Delta)^{\frac{1}{2}}f\|_{L^{2}}^{\frac{1}{2}},
\end{equation}
which the proof can be found in \cite{GG11}.

\vskip.1in

Let us complete this section by presenting the following lemma in terms of pressure and the gradient of pressure which will play  an important role in Section 3.

\begin{lemma}\label{le2.6}
Let $(u,v,w)$ be a smooth solution of the system \eqref{eq1.1}, $\Pi$ and $\Psi$ are determined by the momentum (first) equations and the Poisson (fifth) equation of \eqref{eq1.1}.  Then for any $1<q<\infty$,  we have
\begin{equation}\label{eq2.7}
  \|\Pi\|_{L^{q}}\leq
  C(\|u\|_{L^{2q}}^{2}+\|\nabla \Psi\|_{L^{2q}}^{2})
\end{equation}
and
\begin{equation}\label{eq2.8}
  \|\nabla \Pi\|_{L^{q}}\leq
  C(\||u||\nabla u|\|_{L^{q}}+\||\nabla\Psi||\Delta\Psi|\|_{L^{q}}),
\end{equation}
where $C$ is a constant depending only on $q$.
\end{lemma}

\begin{proof}
Notice that by \eqref{eq1.2}, we actually have
\begin{equation}\label{eq2.9}
   \Delta\Psi\nabla\Psi=\nabla\cdot(\nabla\Psi\otimes\nabla\Psi)-\frac{1}{2}\nabla(|\nabla\Psi|^{2}).
\end{equation}
Taking $\nabla\cdot$ on both sides of the first equations of \eqref{eq1.1},  with the aid of  the divergence free condition $\nabla\cdot u=0$ and the force balance relation \eqref{eq2.9}, one can obtain
\begin{equation}\label{eq2.10}
    -\Delta \Pi=\nabla\cdot((u\cdot\nabla) u-\Delta\Psi\nabla\Psi)=\sum_{i,j=1}^{3}\partial_{x_{i}}\partial_{x_{j}}(u_{i}u_{j}-\partial_{x_{i}}\Psi\partial_{x_{j}}\Psi)+\frac{1}{2}\Delta(|\nabla\Psi|^{2}),
\end{equation}
which yields to
\begin{equation*}
   \Pi=\sum_{i,j=1}^{3}\mathcal{R}_{i}\mathcal{R}_{j}(u_{i}u_{j}-\partial_{x_{i}}\Psi\partial_{x_{j}}\Psi)-\frac{1}{2}|\nabla\Psi|^{2}.
\end{equation*}
Recall that $\mathcal{R}_{i}$ ($i=1,2,3$) are Riesz operators, the Calderon-Zygmund inequality implies that
\begin{equation*}
    \|\Pi\|_{L^{q}}\leq C(\|u\|_{L^{2q}}^{2}+\|\nabla\Psi\|_{L^{2q}}^{2})
\end{equation*}
and
\begin{equation*}
    \|\nabla\Pi\|_{L^{q}}\leq C(\||u||\nabla u|\|_{L^{q}}+\||\nabla\Psi||\Delta\Psi|\|_{L^{q}}).
\end{equation*}
The proof of Lemma \ref{le2.6} is complete.
\end{proof}

\section{A Priori Estimates}

 The proofs of Theorems \ref{th1.1} and \ref{th1.2} are essentially based on the following three a priori estimates for local smooth solutions of the system \eqref{eq1.1}.
\begin{lemma}\label{le3.1}
Let  $u_{0}\in H^{3}(\mathbb{R}^{3})$ with $\nabla\cdot u_{0}=0$, $v_{0}, w_{0}\in L^{1}(\mathbb{R}^{3})\cap
H^{2}(\mathbb{R}^{3})$ with $v_{0}, w_{0}\geq 0$, and let $(u,v,w)$ be the corresponding local smooth solution to the system
\eqref{eq1.1} on $[0,T)$  satisfying the condition
\eqref{eq1.12}. Then we have
\begin{align}\label{eq3.1}
   \sup_{0\leq t<T}(\|u(\cdot,t)\|_{H^{3}}+\|(v(\cdot,t),w(\cdot,t))\|_{H^{2}})\leq C_{0},
\end{align}
where $C_{0}$ is a constant depending only on $\|u_{0}\|_{H^{3}}$, $\|(v_{0}, w_{0})\|_{L^{1}\cap H^{2}}$, $T$.
\end{lemma}
\begin{proof}
We first notice that the masses of $v$ and $w$ are conserved, i.e., for every $t\in[0,T)$, there hold:
\begin{equation*}
  \int_{\mathbb{R}^{3}}v(x,t)dx=\int_{\mathbb{R}^{3}}v_{0}(x)dx\ \ \ \ \text{and}\ \ \ \  \int_{\mathbb{R}^{3}}w(x,t)dx=\int_{\mathbb{R}^{3}}w_{0}(x)dx.
\end{equation*}
Moreover, by the maximum principle, we deduce that if $v_0$ and $w_0$  are non-negative, then $v$ and $w$ are also non-negative, see \cite{S09} for more details.
\vskip .1in

 \textbf{ Step 1:  Basic energy estimates} Testing the third equation of \eqref{eq1.1} by $v$, the fourth equation of \eqref{eq1.1} by $w$, summing up the resulting equalities, and using the fifth equation of \eqref{eq1.1} and the divergence free condition, one can easily obtain
\begin{align}\label{eq3.2}
  \frac{1}{2}\frac{d}{dt}\big(\|v\|_{L^{2}}^{2}+\|w\|_{L^{2}}^{2}\big)+\|\nabla
  v\|_{L^{2}}^{2}+\|\nabla
  w\|_{L^{2}}^{2}&+\int_{\mathbb{R}^{3}}(v+w)(v-w)^{2}
  dx=0.
\end{align}
Since $v$ and $w$ are non-negative, we infer from \eqref{eq3.2} that  for all $0\leq t<T$,
\begin{align}\label{eq3.3}
  \|v\|_{L^{2}}^{2}+\|w\|_{L^{2}}^{2}+2\int_{0}^{t}\big(\|\nabla
  v(\tau)\|_{L^{2}}^{2}+\|\nabla
  w(\tau)\|_{L^{2}}^{2}\big)d\tau\leq \|v_{0}\|_{L^{2}}^{2}+\|w_{0}\|_{L^{2}}^{2}.
\end{align}
On the other hand, multiplying
 the first equations of \eqref{eq1.1} by $u$,  the third and fourth equations of  \eqref{eq1.1} by
$\Psi$,  then integrating over $\mathbb{R}^{3}$, with the aid of the fifth equation of \eqref{eq1.1}, it can be easily seen
that
\begin{align}\label{eq3.4}
  &\frac{1}{2}\frac{d}{dt}\|u\|_{L^{2}}^{2}+\|\nabla
  u\|_{L^{2}}^{2}=\int_{\mathbb{R}^{3}}(v-w)u\cdot\nabla\Psi dx,\\
\label{eq3.5}
  &\int_{\mathbb{R}^{3}}\big(\partial_{t}v\Psi+\nabla\cdot(v\nabla\Psi)\Psi-\Delta
  v\Psi+(u\cdot\nabla)
  v\Psi\big)dx=0,\\
\label{eq3.6}
  &\int_{\mathbb{R}^{3}}\big(\partial_{t}w\Psi-\nabla\cdot(w\nabla\Psi)\Psi-\Delta
  w\Psi+(u\cdot\nabla)
  w\Psi\big)dx=0.
\end{align}
Subtracting \eqref{eq3.6} from \eqref{eq3.5} and using the fifth equation of \eqref{eq1.1} again, gives us to
\begin{equation}\label{eq3.7}
  \frac{1}{2}\frac{d}{dt}\|\nabla\Psi\|_{L^{2}}^{2}+\int_{\mathbb{R}^{3}}(v+w)|\nabla\Psi|^{2}dx+\int_{\mathbb{R}^{3}}|\Delta\Psi|^{2}dx
  +\int_{\mathbb{R}^{3}}(v-w)u\cdot\nabla\Psi dx
  =0.
\end{equation}
Adding \eqref{eq3.4} and \eqref{eq3.7} together, we find that
\begin{equation}\label{eq3.8}
   \frac{1}{2}\frac{d}{dt}(\|u\|_{L^{2}}^{2}+\|\nabla\Psi\|_{L^{2}}^{2})+\|\nabla
  u\|_{L^{2}}^{2}+\|\Delta\Psi\|_{L^{2}}^{2}
  +\int_{\mathbb{R}^{3}}(v+w)|\nabla\Psi|^{2}dx=0,
\end{equation}
which yields the $L^{2}$ bound of $u$ and $\nabla\Psi$ by observing the non-negativity of $v$ and $w$:
\begin{equation}\label{eq3.9}
  \|u\|_{L^{2}}^{2}+\|\nabla\Psi\|_{L^{2}}^{2}+2\int_{0}^{t}\big(\|\nabla
  u(\tau)\|_{L^{2}}^{2}+\|\Delta\Psi(\tau)\|_{L^{2}}^{2}\big)d\tau\leq
  C
\end{equation}
 for all $0\leq t<T$, where $C$ is a constant depending only on $\|u_{0}\|_{L^{2}}$ and $\|(v_{0},
w_{0})\|_{L^{1}\cap L^{2}}$.
\vskip .1in

\textbf{ Step 2:  $L^{4}$ estimate of the velocity}
Multiplying the first equations of \eqref{eq1.1} by $|u|^{2}u$ and integrating over $\mathbb{R}^{3}$, after suitable integration by parts, one has
\begin{align}\label{eq3.10}
  \frac{1}{4}\frac{d}{dt}\|u\|_{L^{4}}^{4}+\||u||\nabla u|\|_{L^{2}}^{2}+\frac{1}{2}\|\nabla
  |u|^{2}\|_{L^{2}}^{2}&\leq -\int_{\mathbb{R}^{3}}\nabla\Pi\cdot u|u|^{2}dx+\int_{\mathbb{R}^{3}}\Delta\Psi\nabla\Psi\cdot u|u|^{2}dx\nonumber\\
  &:=I_{1}+I_{2}.
\end{align}
Notice that, by choosing $q=2$ in \eqref{eq2.7} and applying \eqref{eq3.3} and \eqref{eq3.9},  one has
\begin{align*}
    \|\Pi\|_{L^{2}}&\leq C\big(\|u\|_{L^{4}}^{2}+\|\nabla\Psi\|_{L^{4}}^{2}\big)\leq C\big(\|u\|_{L^{4}}^{2}+\|\nabla\Psi\|_{L^{2}}^{\frac{1}{2}}\|\Delta\Psi\|_{L^{2}}^{\frac{3}{2}}\big)\nonumber\\
    &\leq C\big(1+\|u\|_{L^{4}}^{2}+\|(v,w)\|_{L^{2}}^{\frac{3}{2}}\big)\leq C\big(1+\|u\|_{L^{4}}^{2}\big),
\end{align*}
where we have used the
following interpolation inequality:
\begin{equation*}
  \|f\|_{L^{4}}\leq \|f\|_{L^{2}}^{\frac{1}{4}}\|
  \nabla f\|_{L^{2}}^{\frac{3}{4}}.
\end{equation*}
Thus, after integration by parts,  by applying Lemma \ref{le2.2} with $p=4$,  $I_{1}$ can be bounded by
\begin{align}\label{eq3.11}
  I_{1}&= \int_{\mathbb{R}^{3}}\Pi\nabla\cdot(u|u|^{2})dx\nonumber\\
  &\leq C  \int_{\mathbb{R}^{3}}|\Pi||u|^{2}|\nabla u|dx\nonumber\\
  &\leq C\|\Pi\|_{L^{4}}\|u\|_{L^{4}}\||u||\nabla u|\|_{L^{2}}\nonumber\\
  &\leq \frac{1}{2}\||u||\nabla u|\|_{L^{2}}^{2}+C\|\Pi\|_{L^{4}}^{2}\|u\|_{L^{4}}^{2}\nonumber\\
  &\leq  \frac{1}{2}\||u||\nabla u|\|_{L^{2}}^{2}+C\|\Pi\|_{L^{2}}\|\Pi\|_{BMO}\|u\|_{L^{4}}^{2}\nonumber\\
  &\leq  \frac{1}{2}\||u||\nabla u|\|_{L^{2}}^{2}+C\|\Pi\|_{BMO}\big(1+\|u\|_{L^{4}}^{4}\big).
\end{align}
For $I_{2}$,  by using  the fifth equation of \eqref{eq1.1} and the following Hardy-Littlewood-Sobolev inequality with $p=2$:
\begin{equation*}
  \|\nabla(-\Delta)^{-1}f\|_{L^{\frac{3p}{3-p}}}\leq
  C\|f\|_{L^{p}} \ \ \text{for any}\ \  1<p<3,
\end{equation*}
we obtain
\begin{align}\label{eq3.12}
  I_{2} &=\int_{\mathbb{R}^{3}}(v-w)\nabla\Psi\cdot u|u|^{2}dx\nonumber\\
  &\leq \|(v,w)\|_{L^{3}}\|\nabla\Psi\|_{L^{6}}\|u|u|^{2}\|_{L^{2}}\nonumber\\
  &\leq C\|(v,w)\|_{L^{2}}^{\frac{3}{2}}\|(\nabla v,\nabla w)\|_{L^{2}}^{\frac{1}{2}}\||u|^{2}\|_{L^{3}}^{\frac{3}{2}}\nonumber\\
  &\leq C\|(\nabla v,\nabla w)\|_{L^{2}}^{\frac{1}{2}}\||u|^{2}\|_{L^{2}}^{\frac{3}{4}}\|\nabla |u|^{2}\|_{L^{2}}^{\frac{3}{4}}\nonumber\\
  &\leq   \frac{1}{4}\|\nabla |u|^{2}\|_{L^{2}}^{2}+C\|(\nabla v,\nabla w)\|_{L^{2}}^{\frac{4}{5}}\|u\|_{L^{4}}^{\frac{12}{5}}\nonumber\\
  &\leq   \frac{1}{4}\|\nabla |u|^{2}\|_{L^{2}}^{2}+C\big(1+\|(\nabla v,\nabla w)\|_{L^{2}}^{2}\big)\big(1+\|u\|_{L^{4}}^{4}\big),
\end{align}
where we have used the interpolation inequality:
\begin{equation*}
  \|f\|_{L^{3}}\leq \|f\|_{L^{2}}^{\frac{1}{2}}\|
  \nabla f\|_{L^{2}}^{\frac{1}{2}}.
\end{equation*}
Since $H^{2}(\mathbb{R}^{3})\hookrightarrow L^{\infty}(\mathbb{R}^{3})\hookrightarrow BMO\hookrightarrow\dot{B}^{0}_{\infty,\infty}(\mathbb{R}^{3})$, we have
\begin{align*}
    \|\Pi\|_{\dot{B}^{0}_{\infty,\infty}}\leq \|\Pi\|_{BMO}\leq \|\Pi\|_{L^{\infty}}\leq C\|\Pi\|_{H^{2}}.
\end{align*}
Moreover, by \eqref{eq2.10}, \eqref{eq3.3} and \eqref{eq3.9}, one has
\begin{align*}
    \|\Pi\|_{H^{2}}&\leq C(\||u|^{2}\|_{H^{2}}+\||\nabla\Psi|^{2}\|_{H^{2}}\big)\nonumber\\
    &\leq C\big(\|u\|_{L^{\infty}}\|u\|_{H^{2}}+\|\nabla\Psi\|_{L^{\infty}}\|\nabla\Psi\|_{H^{2}}\big)\nonumber\\
    &\leq C\big(\|u\|_{H^{2}}^{2}+\|\nabla\Psi\|_{H^{2}}^{2}\big)\nonumber\\
    &\leq C\big(\|u\|_{H^{2}}^{2}+\|\nabla\Psi\|_{L^{2}}^{2}+\|\Delta\Psi\|_{L^{2}}^{2}+\|\nabla\Delta\Psi\|_{L^{2}}^{2}\big)\nonumber\\
    &\leq C\big(1+\|u\|_{H^{2}}^{2}+\|(v,w)\|_{H^{1}}^{2}\big)\nonumber\\
    &\leq C\big(1+\|\nabla\Delta u\|_{L^{2}}^{2}+\|(\Delta v,\Delta w)\|_{L^{2}}^{2}\big).
\end{align*}
Plugging \eqref{eq3.11} and \eqref{eq3.12} into \eqref{eq3.10},  and using the above facts and Lemma \ref{le2.3},  it follows that
\begin{align}\label{eq3.13}
  \frac{d}{dt}\|u\|_{L^{4}}^{4}&\leq C\Big(1+\|(\nabla v,\nabla w)\|_{L^{2}}^{2}+\|\Pi\|_{BMO}\Big)\Big(1+\|u\|_{L^{4}}^{4}\Big)\nonumber\\
  &\leq   C\Big(1+\|(\nabla v,\nabla w)\|_{L^{2}}^{2}+\|\Pi\|_{\dot{B}^{0}_{\infty,\infty}}\sqrt{1+\ln(e+\|\Pi\|_{H^{2}})}\Big)\Big(1+\|u\|_{L^{4}}^{4}\Big)\nonumber\\
  &\leq
  C\Big(1+\|(\nabla v,\nabla w)\|_{L^{2}}^{2}+\frac{\|\Pi\|_{\dot{B}^{0}_{\infty,\infty}}}{\sqrt{1+\ln(e+\|\Pi\|_{\dot{B}^{0}_{\infty,\infty}})}}
  \Big)\nonumber\\
  &\ \ \ \ \ \ \ \times\Big(1+\ln(e+\|\Pi\|_{H^{2}})\Big)\Big(1+\|u\|_{L^{4}}^{4}\Big)\nonumber\\
  &\leq
  C\Big(1+\|(\nabla v,\nabla w)\|_{L^{2}}^{2}+\frac{\|\Pi\|_{\dot{B}^{0}_{\infty,\infty}}}{\sqrt{1+\ln(e+\|\Pi\|_{\dot{B}^{0}_{\infty,\infty}})}}
  \Big)\nonumber\\
  &\ \ \ \ \ \ \ \times\ln\Big(e+\|\nabla\Delta u\|_{L^{2}}^{2}+\|(\Delta v, \Delta w)\|_{L^{2}}^{2}\Big)\Big(1+\|u\|_{L^{4}}^{4}\Big).
\end{align}
By the fact \eqref{eq3.3} and the condition \eqref{eq1.12}, one concludes that for any small
constant $\sigma>0$, there exists $T_{0}<T$ such that
\begin{equation*}
  \int_{T_{0}}^{T}\big(1+\|(\nabla v,\nabla w)\|_{L^{2}}^{2}\big)dt<\sigma
\end{equation*}
 and
\begin{equation*}
  \int_{T_{0}}^{T}\frac{\|\Pi\|_{\dot{B}^{0}_{\infty,\infty}}}{\sqrt{1+\ln(e+\|\Pi\|_{\dot{B}^{0}_{\infty,\infty}})}}dt<\sigma.
\end{equation*}
Setting
\begin{equation*}\label{eq3.14}
  Y(t):=\sup_{T_{0}\leq \tau\leq t}\big(\|\nabla\Delta
  u(\cdot,\tau)\|_{L^{2}}^{2}+\|(\Delta v(\cdot,\tau), \Delta w(\cdot,\tau))\|_{L^{2}}^{2}\big).
\end{equation*}
Then applying Gronwall's inequality to \eqref{eq3.13} in the time
interval $[T_{0},t]$ for any $T_{0}\leq t<T$, we conclude that
\begin{align}\label{eq3.14}
 \|u(t)\|_{L^{4}}^{4}&\leq C_{1}\exp\Big(\int_{T_{0}}^{t}C\big(1\!+\!\|(\nabla v,\nabla w)\|_{L^{2}}^{2}\!+\!\frac{\|\Pi\|_{\dot{B}^{0}_{\infty,\infty}}}{\sqrt{1+\ln(e+\|\Pi\|_{\dot{B}^{0}_{\infty,\infty}})}}\big)d\tau\ln(e\!+\!Y(t))
  \Big)\nonumber\\
  &\leq C_{1}\exp\big(2C\sigma\ln(e+Y(t))\big)\nonumber\\
  &\leq C_{1}(e+Y(t))^{2C\sigma},
\end{align}
where $C_{1}=1+\|u(\cdot,T_{0})\|_{L^{4}}^{4}$.
\vskip .1in

\textbf{ Step 3:  $L^{2}$ estimate for the gradient of velocity  field}
Multiplying  the first equations of \eqref{eq1.1}  by $-\Delta u$, and integrating over $\mathbb{R}^{3}$, one has
\begin{align}\label{eq3.15}
  \frac{1}{2}\frac{d}{dt}\|\nabla u\|_{L^{2}}^{2}+\|\Delta u\|_{L^{2}}^{2}&=\int_{\mathbb{R}^{3}}(u\cdot\nabla) u\cdot \Delta u dx-\int_{\mathbb{R}^{3}}\Delta\Psi\nabla \Psi\cdot\Delta u dx\nonumber\\
  &:=J_{1}+J_{2}.
\end{align}
By applying Young's inequality, $J_{1}$ can be estimated as follows:
\begin{align}\label{eq3.16}
   J_{1}&\leq
  \|u\|_{L^{4}}\|\nabla u\|_{L^{4}}\|\Delta u\|_{L^{2}}\nonumber\\
  &\leq \frac{1}{8}\|\Delta u\|_{L^{2}}^{2}+C \|u\|_{L^{4}}^{2}\|\nabla u\|_{L^{4}}^{2}\nonumber\\
  &\leq \frac{1}{8}\|\Delta u\|_{L^{2}}^{2}+C \|u\|_{L^{4}}^{2}\|u\|_{L^{2}}^{\frac{1}{4}}\|\Delta u\|_{L^{2}}^{\frac{7}{4}}
 \nonumber\\
  &\leq \frac{1}{4}\|\Delta u\|_{L^{2}}^{2}+C \|u\|_{L^{4}}^{16},
\end{align}
where we have used the
 interpolation inequality:
\begin{equation*}
  \|\nabla f\|_{L^{4}}\leq \|f\|_{L^{2}}^{\frac{1}{8}}\|
  \Delta f\|_{L^{2}}^{\frac{7}{8}}.
\end{equation*}
Moreover, applying the fifth equation of \eqref{eq1.1},  \eqref{eq3.3} and \eqref{eq3.9},  we bound $J_{2}$ as
\begin{align}\label{eq3.17}
  J_{2}  &\leq C \|\nabla \Psi\|_{L^{4}}\|\Delta \Psi\|_{L^{4}}\|\Delta u\|_{L^{2}}\nonumber\\
  &\leq \frac{1}{4}\|\Delta u\|_{L^{2}}^{2}+C\|(v,w)\|_{L^{4}}^{2}\|\nabla \Psi\|_{L^{4}}^{2}\nonumber\\
  &\leq\frac{1}{4}\|\Delta u\|_{L^{2}}^{2}+C\|(v,w)\|_{L^{2}}^{\frac{1}{2}}\|(\nabla v,\nabla w)\|_{L^{2}}^{\frac{3}{2}}\|\nabla\Psi\|_{L^{2}}^{\frac{1}{2}}\|\Delta\Psi\|_{L^{2}}^{\frac{3}{2}}\nonumber\\
  &\leq \frac{1}{4}\|\Delta u\|_{L^{2}}^{2}+C\|(v,w)\|_{L^{2}}^{2}\|(\nabla v,\nabla w)\|_{L^{2}}^{\frac{3}{2}}\|\nabla\Psi\|_{L^{2}}^{\frac{1}{2}}\nonumber\\
  &\leq \frac{1}{4}\|\Delta u\|_{L^{2}}^{2}+C\|(v,w)\|_{L^{2}}^{2}\|(\nabla v,\nabla w)\|_{L^{2}}^{2}+C\|(v,w)\|_{L^{2}}^{2}\|\nabla\Psi\|_{L^{2}}^{2}\nonumber\\
   &\leq \frac{1}{4}\|\Delta u\|_{L^{2}}^{2}+C(\|(\nabla v,\nabla w)\|_{L^{2}}^{2}+1).
\end{align}
Plugging \eqref{eq3.16} and \eqref{eq3.17} into \eqref{eq3.15}, we
obtain
\begin{equation}\label{eq3.18}
  \frac{d}{dt}\|\nabla u\|_{L^{2}}^{2}+\|\Delta u\|_{L^{2}}^{2}\leq
  C(\|u\|_{L^{4}}^{16}+\|(\nabla v,\nabla w)\|_{L^{2}}^{2}+1).
\end{equation}
Integrating \eqref{eq3.18} on  the time interval $[T_{0},t]$ and using \eqref{eq3.3} and \eqref{eq3.14},  it follows that
\begin{align}\label{eq3.19}
  \|\nabla u(t)\|_{L^{2}}^{2}&+\int_{T_{0}}^{t}\|\Delta u(\tau)\|_{L^{2}}^{2}d\tau
  \leq C\int_{T_{0}}^{t}\big(\|u\|_{L^{4}}^{16}+\|(\nabla v,\nabla w)\|_{L^{2}}^{2}+1\big)d\tau+\|\nabla u(T_{0})\|_{L^{2}}^{2}\nonumber\\
  &\leq
  C+C(1+C_{1}^{4}(e+Y(t))^{8C\sigma})(t-T_{0})+\|\nabla u(T_{0})\|_{L^{2}}^{2}\nonumber\\
  &\leq C(1+C_{1}^{4}(e+Y(t))^{8C\sigma}),
\end{align}
where $C$ is a constant depending only on $\|u_{0}\|_{H^{1}}$, $\|(v_{0},
w_{0})\|_{L^{1}\cap L^{2}}$ and $T$.
\vskip .1in

\textbf{ Step 4:  Higher-order derivatives estimates}
Taking $\nabla\Delta$ on the first equations of \eqref{eq1.1}, multiplying the
resulting equality with $\nabla\Delta u$ and integrating over $\mathbb{R}^{3}$, observing
that the pressure $\Pi$ can be eliminated by the incompressible condition
$\nabla\cdot u=0$, one obtains that
\begin{align}\label{eq3.20}
  \frac{1}{2}\frac{d}{dt}\|\nabla\Delta
  u\|_{L^{2}}^{2}+\|\Delta^{2}u\|_{L^{2}}^{2}&=-\int_{\mathbb{R}^{3}}\nabla\Delta((u\cdot\nabla)
  u)\cdot\nabla\Delta u dx+\int_{\mathbb{R}^{3}}\nabla\Delta(\Delta\Psi\nabla\Psi)\cdot\nabla\Delta u dx\nonumber\\
  &:=K_{1}+K_{2}.
\end{align}
Since $\nabla\cdot u=0$, it follows from Lemma \ref{le2.4} that
\begin{align}\label{eq3.21}
  K_{1}&=-\int_{\mathbb{R}^{3}}\big[\nabla\Delta((u\cdot\nabla)
  u)-(u\cdot\nabla)\nabla\Delta u\big]\cdot\nabla\Delta u dx\nonumber\\
  &\leq C\|\nabla\Delta((u\cdot\nabla)
  u)-(u\cdot\nabla)\nabla\Delta u\|_{L^{\frac{4}{3}}}\|\nabla\Delta
  u\|_{L^{4}}\nonumber\\
  &\leq C\|\nabla u\|_{L^{2}}\|\nabla\Delta
  u\|_{L^{4}}^{2}\nonumber\\
  &\leq C\|\nabla u\|_{L^{2}}^{\frac{7}{6}}\|\Delta^{2}
  u\|_{L^{2}}^{\frac{11}{6}}\nonumber\\
  &\leq \frac{1}{4}\|\Delta^{2}
  u\|_{L^{2}}^{2} +C\|\nabla u\|_{L^{2}}^{14},
\end{align}
where we have used the interpolation inequality:
\begin{equation*}
    \|\nabla\Delta u\|_{L^{4}}\leq C\|\nabla u\|_{L^{2}}^{\frac{1}{12}}\|\Delta^{2}
  u\|_{L^{2}}^{\frac{11}{12}}.
\end{equation*}
For $K_{2}$, applying the fifth equation of \eqref{eq1.1}, \eqref{eq3.3} and \eqref{eq3.9}, it follows from the Leibniz's law that
\begin{align}\label{eq3.22}
  &K_{2}=-\int_{\mathbb{R}^{3}}\Delta((v-w)\nabla\Psi)\cdot\Delta^{2}u dx\nonumber\\
  &\leq \frac{1}{4}\|\Delta^{2}u\|_{L^{2}}^{2}+C\|\Delta((v-w)\nabla\Psi)\|_{L^{2}}^{2}\nonumber\\
  &\leq \frac{1}{4}\|\Delta^{2}u\|_{L^{2}}^{2}+C\big(\|(\Delta v-\Delta w)\nabla\Psi\|_{L^{2}}^{2}+2\|(\nabla v-\nabla w)\nabla^{2}\Psi\|_{L^{2}}^{2}+\|(v-w)\nabla\Delta\Psi\|_{L^{2}}^{2}\big)\nonumber\\
  &\leq \frac{1}{4}\|\Delta^{2}u\|_{L^{2}}^{2}+C\big(\|(\Delta v, \Delta w)\|_{L^{3}}^{2}\|\nabla\Psi\|_{L^{6}}^{2}+\|(v,w)\|_{L^{3}}^{2}\|(\nabla v,\nabla w)\|_{L^{6}}^{2}\big)\nonumber\\
  &\leq \frac{1}{4}\|\Delta^{2}u\|_{L^{2}}^{2}+C\big(\|(\Delta v, \Delta w)\|_{L^{2}}\|(\nabla\Delta v, \nabla\Delta w)\|_{L^{2}}+\|(\nabla v,\nabla w)\|_{L^{2}}\|(\Delta v, \Delta w)\|_{L^{2}}^{2}\big)\nonumber\\
  &\leq \frac{1}{4}\|\Delta^{2}u\|_{L^{2}}^{2}+\frac{1}{6}\|(\nabla\Delta v, \nabla\Delta w)\|_{L^{2}}^{2}+C\big(1+\|(\nabla v, \nabla w)\|_{L^{2}}^{2}\big)\|(\Delta v, \Delta w)\|_{L^{2}}^{2}.
\end{align}
Putting \eqref{eq3.21} and \eqref{eq3.22} together, we deduce that
\begin{align}\label{eq3.23}
  \frac{d}{dt}\|\nabla\Delta
  u\|_{L^{2}}^{2}&+\|\Delta^{2}u\|_{L^{2}}^{2}\leq \frac{1}{3}\|(\nabla\Delta v, \nabla\Delta w)\|_{L^{2}}^{2}\nonumber\\
  &+C\|\nabla u\|_{L^{2}}^{14}+C\big(1+\|(\nabla v, \nabla w)\|_{L^{2}}^{2}\big)\|(\Delta v, \Delta w)\|_{L^{2}}^{2}.
\end{align}
Taking $\Delta$ to the third equation of \eqref{eq1.1}, then multiplying $\Delta v$, after integration by parts, we see that
\begin{align}\label{eq3.24}
  \frac{1}{2}\frac{d}{dt}\|\Delta
  v\|_{L^{2}}^{2}+\|\nabla\Delta v\|_{L^{2}}^{2}&=-\int_{\mathbb{R}^{3}}\Delta((u\cdot\nabla)
  v)\Delta v dx-\int_{\mathbb{R}^{3}}\Delta\nabla\cdot(v\nabla\Psi)\Delta v dx\nonumber\\
  &:=L_{1}+L_{2}.
\end{align}
Applying \eqref{eq3.3} and \eqref{eq3.9},  $L_{1}$ and $L_{2}$ can be bounded by
\begin{align*}
  L_{1}&=\int_{\mathbb{R}^3}\nabla((u\cdot\nabla) v)\cdot\nabla\Delta vdx\nonumber\\
  &\leq \frac{1}{8}\|\nabla\Delta v\|_{L^{2}}^{2}+C\big(\|(\nabla u\cdot\nabla)v\|_{L^{2}}^{2}+\|(u\cdot\nabla)\nabla v\|_{L^{2}}^{2}\big)\nonumber\\
  &\leq \frac{1}{8}\|\nabla\Delta v\|_{L^{2}}^{2}+C\big(\|\nabla u\|_{L^{\infty}}^{2}\|\nabla v\|_{L^{2}}^{2}+\|u\|_{L^{4}}^{2}\|\Delta v\|_{L^{4}}^{2}\big)\nonumber\\
  &\leq \frac{1}{8}\|\nabla\Delta v\|_{L^{2}}^{2}+C\big(\|\nabla v\|_{L^{2}}^{2}(1+\|\nabla \Delta u\|_{L^{2}}^{2})+\|u\|_{L^{2}}^{\frac{1}{2}}\|\nabla u\|_{L^{2}}^{\frac{3}{2}}\| v\|_{L^{2}}^{\frac{1}{6}}\|\nabla\Delta v\|_{L^{2}}^{\frac{11}{6}}\big)\nonumber\\
  &\leq \frac{1}{6}\|\nabla\Delta v\|_{L^{2}}^{2}+C\|\nabla v\|_{L^{2}}^{2}\big(1+\|\nabla \Delta u\|_{L^{2}}^{2}\big)+C\|\nabla u\|_{L^{2}}^{18};
\end{align*}

\begin{align*}
  L_{2}&=\int_{\mathbb{R}^3}\Delta(v\nabla \Psi)\cdot\nabla\Delta vdx\nonumber\\
  &\leq \frac{1}{8}\|\nabla\Delta v\|_{L^{2}}^{2}+C\big(\|\Delta v\nabla\Psi\|_{L^{2}}^{2}+\|\nabla v\nabla^{2}\Psi\|_{L^{2}}^{2}+\|v\nabla \Delta\Psi\|_{L^{2}}^{2}\big)\nonumber\\
  &\leq \frac{1}{8}\|\nabla\Delta v\|_{L^{2}}^{2}+C\big(\|\Delta v\|_{L^{3}}^{2}\|\nabla \Psi\|_{L^{6}}^{2}+\|\nabla v\|_{L^{6}}^{2}\|\nabla^{2} \Psi\|_{L^{3}}^{2}+\|v\|_{L^{3}}^{2}\|\nabla(v-w)\|_{L^{6}}^{2}\big)\nonumber\\
  &\leq \frac{1}{6}\|\nabla\Delta v\|_{L^{2}}^{2}+C\big(1+\|(\nabla v, \nabla w)\|_{L^{2}}^{2}\big)\big(1+\|(\Delta v, \Delta w)\|_{L^{2}}^{2}\big).
\end{align*}
Taking the above two inequalities into \eqref{eq3.24} gives us to
\begin{align}\label{eq3.25}
  \frac{d}{dt}\|\Delta
  v\|_{L^{2}}^{2}&+\frac{4}{3}\|\nabla\Delta v\|_{L^{2}}^{2}\leq C\|\nabla u\|_{L^{2}}^{18}\nonumber\\
  &+C\big(1+\|(\nabla v, \nabla w)\|_{L^{2}}^{2}\big)\big(1+\|\nabla \Delta u\|_{L^{2}}^{2}+\|(\Delta v, \Delta w)\|_{L^{2}}^{2}\big).
\end{align}
The bound of $w$  can be derived analogously, thus we get
\begin{align}\label{eq3.26}
  \frac{d}{dt}\|\Delta
  w\|_{L^{2}}^{2}&+\frac{4}{3}\|\nabla\Delta w\|_{L^{2}}^{2}\leq C\|\nabla u\|_{L^{2}}^{18}\nonumber\\
  &+C\big(1+\|(\nabla v, \nabla w)\|_{L^{2}}^{2}\big)\big(1+\|\nabla \Delta u\|_{L^{2}}^{2}+\|(\Delta v, \Delta w)\|_{L^{2}}^{2}\big).
\end{align}
Now we  complete the proof of Lemma \ref{le3.1}.  We conclude from
 \eqref{eq3.14},  \eqref{eq3.19}, \eqref{eq3.23}, \eqref{eq3.25} and \eqref{eq3.26} that
\begin{align}\label{eq3.27}
  \frac{d(e+Y(t))}{dt}&\leq C\big(\|\nabla u\|_{L^{2}}^{14}+\|\nabla u\|_{L^{2}}^{18}\big)+C\big(1+\|(\nabla v, \nabla w)\|_{L^{2}}^{2}\big)(e+Y(t))\nonumber\\
  &\leq C(1+C_{1}^{36}(e+Y(t))^{72C\sigma})+C\big(1+\|(\nabla v, \nabla w)\|_{L^{2}}^{2}\big)(e+Y(t)).
\end{align}
Choosing $\sigma$ small enough such that $72C\sigma\leq 1$, the above inequality \eqref{eq3.27} yields to
\begin{align}\label{eq3.28}
  \frac{d}{dt}(e+Y(t))\leq C\big(1+\|(\nabla v, \nabla w)\|_{L^{2}}^{2}\big)(e+Y(t)).
\end{align}
For any $T_{0}\leq t<T$, applying Gronwall's inequality to \eqref{eq3.28},   we get
\begin{align*}
  Y(t)\leq (e+Y(T_{0}))\exp\Big(C\int_{T_{0}}^{t}\big(1+\|(\nabla v, \nabla w)\|_{L^{2}}^{2}\big)d\tau
  \Big),
\end{align*}
where $Y(T_{0})=e+\|\nabla\Delta u(\cdot,T_{0})\|_{L^{2}}^{2}+\|(\Delta v(\cdot,T_{0}),\Delta w(\cdot,T_{0}))\|_{L^{2}}^{2}$. This together with the basic energy estimates \eqref{eq3.3} and \eqref{eq3.9} yield
that
\begin{equation*}
  \sup_{0\leq t<T}\big(\|u(\cdot,t)\|_{H^{3}}+\|(v(\cdot,t),w(\cdot,t))\|_{H^{2}}\big)\leq
  C_0,
\end{equation*}
where $C_{0}$ is a constant depending only on $\|u_0\|_{H^{3}}$,
$\|(v_0,w_0)\|_{L^{1}\cap H^{2}}$ and $T$.
We complete the proof of Lemma \ref{le3.1}.
\end{proof}

Similarly, under the condition \eqref{eq1.13}, we have the following a priori estimate of local smooth solutions.
\begin{lemma}\label{le3.2}
Under the hypotheses of Lemma \ref{le3.1} if we replace the condition \eqref{eq1.12} by \eqref{eq1.13}, we still have the desired bound \eqref{eq3.1}.
\end{lemma}
\begin{proof}
Since Lemma \ref{le3.2} can be proven analogously to Lemma \ref{le3.1},  we sketch its proof  here.
Notice that Riesz operators are bounded in $L^{p}(\mathbb{R}^{3})$ for all $1<p<\infty$, it follows from Lemma \ref{le2.6} that
\begin{align}\label{eq3.29}
\|\nabla\Pi\|_{L^{2}} &\leq C\big(\||u||\nabla u|\|_{L^{2}}+\||\nabla\Psi||\Delta\Psi|\|_{L^{2}}\big)\nonumber\\
&\leq C\big(\||u||\nabla u|\|_{L^{2}}+\|\nabla\Psi\|_{L^{6}}\|(v,w)\|_{L^{3}}\big)\nonumber\\
&\leq C\big(\||u||\nabla u|\|_{L^{2}}+\|(v, w)\|_{L^{2}}^{\frac{3}{2}}\|(\nabla v,\nabla w)\|_{L^{2}}^{\frac{1}{2}}\big)\nonumber\\
&\leq C\big(\||u||\nabla u|\|_{L^{2}}+\|(\nabla v,\nabla w)\|_{L^{2}}^{\frac{1}{2}}\big),
\end{align}
which leads to the bound of $I_{1}$  by \eqref{eq2.6},
\begin{align}\label{eq3.30}
  I_{1} &\leq \frac{1}{4}\||u||\nabla u|\|_{L^{2}}^{2}+C\|\Pi\|_{L^{4}}^{2}\|u\|_{L^{4}}^{2}\nonumber\\
  &\leq  \frac{1}{4}\||u||\nabla u|\|_{L^{2}}^{2}+C\|\Pi\|_{\dot{B}^{-1}_{\infty,\infty}}\|\nabla\Pi\|_{L^{2}}\|u\|_{L^{4}}^{2}\nonumber\\
  &\leq  \frac{1}{4}\||u||\nabla u|\|_{L^{2}}^{2}+C\|\Pi\|_{\dot{B}^{-1}_{\infty,\infty}}\big(\||u||\nabla u|\|_{L^{2}}+\|(\nabla v,\nabla w)\|_{L^{2}}^{\frac{1}{2}}\big)\|u\|_{L^{4}}^{2}\nonumber\\
  &\leq  \frac{1}{2}\||u||\nabla u|\|_{L^{2}}^{2}+C\big(1+\|(\nabla v,\nabla w)\|_{L^{2}}^{2}+\|\Pi\|_{\dot{B}^{-1}_{\infty,\infty}}^{2}\big)\big(1+\|u\|_{L^{4}}^{4}\big).
\end{align}
Since $H^{2}(\mathbb{R}^{3})\hookrightarrow L^{3}(\mathbb{R}^{3})\hookrightarrow\dot{B}^{-1}_{\infty,\infty}(\mathbb{R}^{3})$, we have
\begin{align*}
    \|\Pi\|_{\dot{B}^{-1}_{\infty,\infty}}\leq  C\|\Pi\|_{L^{3}}\leq C\|\Pi\|_{H^{2}}.
\end{align*}
Hence, combining  \eqref{eq3.30} with \eqref{eq3.12},  it follows that
\begin{align}\label{eq3.31}
  &\frac{d}{dt}\|u\|_{L^{4}}^{4}\leq C\Big(1+\|(\nabla v,\nabla w)\|_{L^{2}}^{2}+\|\Pi\|_{\dot{B}^{-1}_{\infty,\infty}}^{2}\Big)\Big(1+\|u\|_{L^{4}}^{4}\Big)\nonumber\\
  &\leq
  C\Big(1+\|(\nabla v,\nabla w)\|_{L^{2}}^{2}+\frac{\|\Pi\|_{\dot{B}^{-1}_{\infty,\infty}}^{2}}{1+\ln(e+\|\Pi\|_{\dot{B}^{-1}_{\infty,\infty}})}
  \Big)\Big(1+\ln(e+\|\Pi\|_{\dot{B}^{-1}_{\infty,\infty}})\Big)\Big(1+\|u\|_{L^{4}}^{4}\Big) \nonumber\\
  &\leq
  C\Big(1+\|(\nabla v,\nabla w)\|_{L^{2}}^{2}+\frac{\|\Pi\|_{\dot{B}^{-1}_{\infty,\infty}}^{2}}{1+\ln(e+\|\Pi\|_{\dot{B}^{-1}_{\infty,\infty}})}
  \Big)\nonumber\\
  &\ \ \ \ \ \ \ \ \times\ln\Big(e+\|\nabla\Delta u\|_{L^{2}}^{2}+\|(\Delta v, \Delta w)\|_{L^{2}}^{2}\Big)\Big(1+\|u\|_{L^{4}}^{4}\Big).
\end{align}
Under the condition \eqref{eq1.13}, one concludes that for any small
constant $\sigma>0$, there exists $T_{0}<T$ such that
\begin{equation*}
  \int_{T_{0}}^{T}\frac{\|\Pi\|_{\dot{B}^{-1}_{\infty,\infty}}^{2}}{1+\ln(e+\|\Pi\|_{\dot{B}^{-1}_{\infty,\infty}})}dt<\sigma.
\end{equation*}
Applying Gronwall's inequality to \eqref{eq3.31} in the time
interval $[T_{0},t]$ for any $T_{0}\leq t< T$, we obtain
\begin{align}\label{eq3.32}
   \|u(t)\|_{L^{4}}^{4}&\leq C_{1}\exp\Big(C\int_{T_{0}}^{t}\big(1+\|(\nabla v,\nabla w)\|_{L^{2}}^{2}+\frac{\|\Pi\|_{\dot{B}^{-1}_{\infty,\infty}}^{2}}{1+\ln(e+\|\Pi\|_{\dot{B}^{0}_{\infty,\infty}})}\big)d\tau\ln(e+Y(t))
   \Big)\nonumber\\
   &\leq C_{1}\exp\big(2C\sigma\ln(e+Y(t))\big)\nonumber\\
   &\leq C_{1}(e+Y(t))^{2C\sigma}.
\end{align}

\vskip.1in

Having obtained \eqref{eq3.32}, we can follow a similar argument
as that in the proof of Lemma \ref{le3.1} to get the desired assertion \eqref{eq3.1},
thus we safely omit the proof here. This completes the proof of Lemma
\ref{le3.2}.
\end{proof}

Furthermore, under the condition \eqref{eq1.14}, we have the following a priori estimate of local smooth solutions.

\begin{lemma}\label{le3.3}
Under the hypotheses of Lemma \ref{le3.1} if we replace the condition \eqref{eq1.12} by \eqref{eq1.14}, we still have the desired bound \eqref{eq3.1}.
\end{lemma}
\begin{proof}
Under the condition \eqref{eq1.14}, by \eqref{eq2.2} and \eqref{eq3.29},  we can bound $I_{1}$ as follows:
\begin{align}\label{eq3.33}
  I_{1} &=-\int_{\mathbb{R}^{3}}\nabla\Pi\cdot |u|^{2}udx\leq \|\nabla \Pi\|_{L^{4}}\|u\|_{L^{4}}^{3}\leq  C\|\nabla \Pi\|_{L^{2}}^{\frac{1}{2}}\|\nabla \Pi\|_{BMO}^{\frac{1}{2}}\|u\|_{L^{4}}^{3}\nonumber\\
   &  \leq C\|\nabla \Pi\|_{BMO}^{\frac{1}{2}}\big(\||u||\nabla u|\|_{L^{2}}+\|(\nabla v,\nabla w)\|_{L^{2}}^{\frac{1}{2}}\big)^{\frac{1}{2}}\|u\|_{L^{4}}^{3}\nonumber\\
   &  \leq C\|\nabla \Pi\|_{BMO}^{\frac{1}{2}}\big(\||u||\nabla u|\|_{L^{2}}^{\frac{1}{2}}+\|(\nabla v,\nabla w)\|_{L^{2}}^{\frac{1}{4}}\big)\|u\|_{L^{4}}^{3}\nonumber\\
   &\leq \frac{1}{2}\||u||\nabla u|\|_{L^{2}}^{2}+C\big(1+\|(\nabla v,\nabla w)\|_{L^{2}}^{2}+\|\nabla \Pi\|_{BMO}^{\frac{2}{3}}\big)\big(1+\|u\|_{L^{4}}^{4}\big).
\end{align}
Putting  \eqref{eq3.33} and \eqref{eq3.12} together,  and using \eqref{eq2.3},  we see that
\begin{align}\label{eq3.34}
  \frac{d}{dt}\|u\|_{L^{4}}^{4}&\leq C\Big(1+\|(\nabla v,\nabla w)\|_{L^{2}}^{2}+\|\nabla \Pi\|_{BMO}^{\frac{2}{3}}\Big)\Big(1+\|u\|_{L^{4}}^{4}\Big)\nonumber\\
  &\leq   C\Big(1+\|(\nabla v,\nabla w)\|_{L^{2}}^{2}+\|\nabla\Pi\|_{\dot{B}^{0}_{\infty,\infty}}^{\frac{2}{3}}\ln^{\frac{1}{3}}(e+\|\nabla\Pi\|_{H^{2}})\Big)\Big(1+\|u\|_{L^{4}}^{4}\Big)\nonumber\\
  &\leq
  C\Big(1+\|(\nabla v,\nabla w)\|_{L^{2}}^{2}+\frac{\|\nabla\Pi\|_{\dot{B}^{0}_{\infty,\infty}}^{\frac{2}{3}}}{\big(1+\ln(e+\|\nabla \Pi(\cdot,t)\|_{\dot{B}^{0}_{\infty,\infty}})\big)^{\frac{2}{3}}}
  \Big)\nonumber\\
  &\ \ \ \ \ \ \ \times\Big(1+\ln(e+\|\nabla\Pi\|_{H^{2}})\Big)\Big(1+\|u\|_{L^{4}}^{4}\Big)\nonumber\\
  &\leq
  C\Big(1+\|(\nabla v,\nabla w)\|_{L^{2}}^{2}+\frac{\|\nabla\Pi\|_{\dot{B}^{0}_{\infty,\infty}}^{\frac{2}{3}}}{\big(1+\ln(e+\|\nabla \Pi(\cdot,t)\|_{\dot{B}^{0}_{\infty,\infty}})\big)^{\frac{2}{3}}}
  \Big)\nonumber\\
  &\ \ \ \ \ \ \ \times\ln\Big(e+\|\nabla\Delta u\|_{L^{2}}^{2}+\|(\Delta v, \Delta w)\|_{L^{2}}^{2}\Big)\Big(1+\|u\|_{L^{4}}^{4}\Big).
\end{align}
Under the condition \eqref{eq1.14}, we conclude that for any small
constant $\sigma>0$, there exists $T_{0}<T$ such that
\begin{equation*}
  \int_{T_{0}}^{T}\frac{\|\nabla \Pi(\cdot,t)\|_{\dot{B}^{0}_{\infty,\infty}}^{\frac{2}{3}}}{\big(1+\ln(e+\|\nabla \Pi(\cdot,t)\|_{\dot{B}^{0}_{\infty,\infty}})\big)^{\frac{2}{3}}}dt<\sigma.
\end{equation*}
The Gronwall's inequality leads \eqref{eq3.34} to
\begin{align}\label{eq3.35}
   \|u(t)\|_{L^{4}}^{4}&\leq C_{1}\exp\Big(C\int_{T_{0}}^{t}\big(1+\|(\nabla v,\nabla w)\|_{L^{2}}^{2}+\frac{\|\nabla \Pi\|_{\dot{B}^{0}_{\infty,\infty}}^{\frac{2}{3}}}{\big(1+\ln(e+\|\nabla \Pi\|_{\dot{B}^{0}_{\infty,\infty}})\big)^{\frac{2}{3}}}\big)d\tau\nonumber\\
   &\ \ \ \ \ \ \ \ \times \ln(e+Y(t))
   \Big)\nonumber\\
   &\leq C_{1}\exp\big(2C\sigma\ln(e+Y(t))\big)\nonumber\\
   &\leq C_{1}(e+Y(t))^{2C\sigma}.
\end{align}
\vskip.1in
Having obtained \eqref{eq3.35}, similarly, we can follow the remaining steps in Lemma \ref{le3.1} to complete the proof of Lemma \ref{le3.3}.
\end{proof}

\section{Proofs of Theorems \ref{th1.1} and \ref{th1.2}}

We are now in a position to prove Theorems \ref{th1.1} and \ref{th1.2}. Since $u_{0}\in H^{3}(\mathbb{R}^{3})$ with $\nabla\cdot u_{0}=0$, $v_{0}, w_{0}\in L^{1}(\mathbb{R}^{3})\cap H^{2}(\mathbb{R}^{3})$,  it follows from the local existence theorem in \cite{J02} that  there exists a time $T^{*}$ and a unique solution $(\widetilde{u},\widetilde{v},\widetilde{w})$ on $[0,T^{*})$ with initial data $(u_{0},v_{0},w_{0})$ satisfying
\begin{equation*}
\begin{cases}
  \widetilde{u}\in C([0,T^{*}), H^{3}(\mathbb{R}^{3}))\cap L^{\infty}(0,T^{*}; H^{3}(\mathbb{R}^{3}))\cap  L^{2}(0, T^{*};
  H^{4}(\mathbb{R}^{3})),\\
  \widetilde{v},\widetilde{w}\in C([0,T^{*}), H^{2}(\mathbb{R}^{3}))\cap L^{\infty}(0,T^{*}; H^{2}(\mathbb{R}^{3}))\cap  L^{2}(0, T^{*};
  H^{3}(\mathbb{R}^{3})).
\end{cases}
\end{equation*}
Notice that it follows from \cite{S09} that the weak solution coincide with the above local smooth solution
$$
  \widetilde{u}\equiv u, \ \  \widetilde{v}\equiv v, \ \ \widetilde{w} \equiv w\ \ \ \ \text{on}\ \ \ \ [0,T^{*}).
$$
Thus it is sufficient to show that $T=T^{*}$. Suppose that $T^{*}<T$. Without loss of generality, we may assume that $T^{*}$ is the maximal existence time for $(\widetilde{u},\widetilde{v},\widetilde{w})$. Since $\widetilde{u}\equiv u$, $\widetilde{v}\equiv v$ and $\widetilde{w} \equiv w$ on $[0,T^{*})$, by the assumptions \eqref{eq1.12}--\eqref{eq1.14}, we have
\begin{align*}
\int_{0}^{T^{*}}\frac{\|\widetilde{\Pi}(\cdot,t)\|_{\dot{B}^{0}_{\infty,\infty}}}{\sqrt{1+\ln(e+\|\widetilde{\Pi}(\cdot,t)\|_{\dot{B}^{0}_{\infty,\infty}})}}dt<\infty,
\end{align*}
or
\begin{align*}
\int_{0}^{T^{*}}\frac{\|\widetilde{\Pi}(\cdot,t)\|_{\dot{B}^{-1}_{\infty,\infty}}^{2}}{1+\ln(e+\|\widetilde{\Pi}(\cdot,t)\|_{\dot{B}^{-1}_{\infty,\infty}})}dt<\infty,
\end{align*}
or
\begin{align*}
\int_{0}^{T^{*}}\frac{\|\nabla \widetilde{\Pi}(\cdot,t)\|_{\dot{B}^{0}_{\infty,\infty}}^{\frac{2}{3}}}{\big(1+\ln(e+\|\nabla \widetilde{\Pi}(\cdot,t)\|_{\dot{B}^{0}_{\infty,\infty}})\big)^{\frac{2}{3}}}dt<\infty,
\end{align*}
where $\widetilde{\Pi}$ is determined by $(\widetilde{u},\widetilde{v},\widetilde{w})$ through the first equations of \eqref{eq1.1} ($\widetilde{\Psi}$ is determined by the fifth equation of \eqref{eq1.1}).
Therefore, it follows from Lemmas \ref{le3.1}--\ref{le3.3} that we can extend the existence time of $(\widetilde{u},\widetilde{v},\widetilde{w})$ beyond the time $T^{*}$, which contradicts with the maximality of $T^{*}$. This completes the proof of Theorems \ref{th1.1} and \ref{th1.2}.

\vskip .3in
\section*{Acknowledgements}
This paper is partially supported by the National Natural Science Foundation of China (11501453), the Fundamental Research Funds for the Central Universities (2014YB031) and the Fundamental Research Project of Natural Science in Shaanxi Province--Young Talent Project (2015JQ1004).



\end{document}